\documentclass[11pt]{article}
\usepackage{amsmath}
\usepackage{amstext,amsbsy}
\usepackage{epsfig,multicol}
\usepackage{amsthm}
\usepackage{amssymb,latexsym}
\usepackage{amsfonts}
\usepackage{color}
\usepackage{tikz}
\usepackage{listings}
\usepackage{enumerate}
\usepackage{framed}
\usepackage{caption}
\usepackage{subcaption}
\usetikzlibrary{decorations.pathreplacing}
\usepackage{bm}

\theoremstyle{plain}
\newtheorem{theorem}{Theorem}[section]
\newtheorem{lemma}[theorem]{Lemma}

\newtheorem{proposition}[theorem]{Proposition}

\theoremstyle{definition}
\newtheorem{definition}[theorem]{Definition}

\newtheorem*{properties*}{Properties}

\newenvironment{definition*}[1][Definition]{\begin{trivlist}
\item[\hskip \labelsep {\bfseries #1}]}{\end{trivlist}}

\numberwithin{equation}{section}

\DeclareMathOperator{\maj}{maj}
\DeclareMathOperator{\inv}{inv}
\DeclareMathOperator{\des}{des}
\DeclareMathOperator{\Des}{Des}

\DeclareMathOperator{\E}{\mathbb{E}}
\DeclareMathOperator{\Var}{Var}

\begin{document}
\title{Tree Descent Polynomials: Unimodality and Central Limit Theorem}
\author{
Amy Grady and Svetlana Poznanovi\'c  \\ [6pt]
School of Mathematical and Statistical Sciences\\
Clemson University, Clemson, SC 29634  \\
}
\date{}
\maketitle
\begin{abstract} For a  poset whose Hasse diagram is a rooted plane forest $F$, we consider the corresponding tree descent polynomial $A_F(q)$, which is a generating function of the number of descents of the labelings of $F$.  When the forest is a path, $A_F(q)$ specializes to the classical Eulerian polynomial. We prove that the coefficient sequence of $A_F(q)$ is unimodal and that if $\{T_{n}\}$ is a sequence of trees with $|T_{n}| = n$ and maximal down degree $D_{n} = O(n^{0.5-\epsilon})$ then the number of descents in a labeling of $T_{n}$ is asymptotically normal. \end{abstract}

{\renewcommand{\thefootnote}{} \footnote{\emph{E-mail addresses}:
agrady@clemson.edu (A. Grady), spoznan@clemson.edu (S.~Poznanovi\'c)}}


\section{Introduction} \label{S:introduction}
In this paper we generalize the Eulerian polynomials by considering descent polynomials of rooted forests. Throughout this paper, we will use $F$ to denote a plane rooted forest, whose roots we will draw on top. Let $V(F)$ be the vertex set of $F$. A labeling $w$ of $F$ of size $n$ is a bijection
\[w: V(F) \rightarrow \{1, \dots, n\}.\] The \emph{descent set} of a labeled forest is
\[ \Des(F,w) = \{v \in V(F) : w(v) > w(u), u \text{  is the parent of } v\} \] and its cardinality is denoted by $\des(F,w)$. For example, for the labeled tree $(T,w_{1})$ in Figure~\ref{fig: des poly example},  for the two child-parent pairs $(v_{4}, v_{6})$  and $(v_{5}, v_{6})$ we have $w_{1}(v_{4}) = 6  > 2 = w_{1}(v_{6})$ and  $w_{1}(v_{5}) = 5  > 2 = w_{1}(v_{6})$. So, $\Des(T,w_{1}) = \{ v_{4}, v_{5} \}$.  Similarly, for the second labeling we get $\Des(T,w_{2}) = \{ v_{2}, v_{3}, v_{4}, v_{5} \}$. The roots are never in the descent set because they don't have parents.

Let $\mathcal{W}(F)$ be the set of all $n!$ labelings of the forest $F$.

\begin{definition} For a rooted forest $F$, the {\em descent polynomial} is  \[A_F(q) = \sum_{w \in \mathcal{W}(F)} q^{\des(F,w)}.\]
\end{definition}

\begin{figure}
\centering
\begin{tikzpicture}
\draw (0,0)--(1,2)--(2,4)--(3,2)--(3,0);
\draw (1,2)--(2,0);
\draw [fill] (0,0) circle [radius=0.05];
\draw [fill] (1,2) circle [radius=0.05];
\draw [fill] (2,0) circle [radius=0.05];
\draw [fill] (2,4) circle [radius=0.05];
\draw [fill] (3,2) circle [radius=0.05];
\draw [fill] (3,0) circle [radius=0.05];
\node [below] at (0,0) {$v_1$};
\node [below] at (2,0) {$v_2$};
\node [below] at (3,0) {$v_3$};
\node [right] at (3,2) {$v_4$};
\node [left] at (1,2) {$v_5$};
\node [above] at (2,4) {$v_6$};

\begin{scope}[shift={(5,0)}]
\draw (0,0)--(1,2)--(2,4)--(3,2)--(3,0);
\draw (1,2)--(2,0);
\draw [fill] (0,0) circle [radius=0.05];
\draw [fill] (1,2) circle [radius=0.05];
\draw [fill] (2,0) circle [radius=0.05];
\draw [fill] (2,4) circle [radius=0.05];
\draw [fill] (3,2) circle [radius=0.05];
\draw [fill] (3,0) circle [radius=0.05];
\node [below] at (0,0) {$1$};
\node [below] at (2,0) {$3$};
\node [below] at (3,0) {$4$};
\node [right] at (3,2) {$6$};
\node [left] at (1,2) {$5$};
\node [above] at (2,4) {$2$};
\end{scope}

\begin{scope}[shift={(10,0)}]
\draw (0,0)--(1,2)--(2,4)--(3,2)--(3,0);
\draw (1,2)--(2,0);
\draw [fill] (0,0) circle [radius=0.05];
\draw [fill] (1,2) circle [radius=0.05];
\draw [fill] (2,0) circle [radius=0.05];
\draw [fill] (2,4) circle [radius=0.05];
\draw [fill] (3,2) circle [radius=0.05];
\draw [fill] (3,0) circle [radius=0.05];
\node [below] at (0,0) {$2$};
\node [below] at (2,0) {$5$};
\node [below] at (3,0) {$6$};
\node [right] at (3,2) {$3$};
\node [left] at (1,2) {$4$};
\node [above] at (2,4) {$1$};
\end{scope}

\end{tikzpicture}
\caption{A tree $T$ with vertex set $V = \{v_{1}, v_{2}, \dots, v_{6}\}$ and two of its labelings, $w_{1}$ and $w_{2}$.} 
\label{fig: des poly example}
\end{figure}
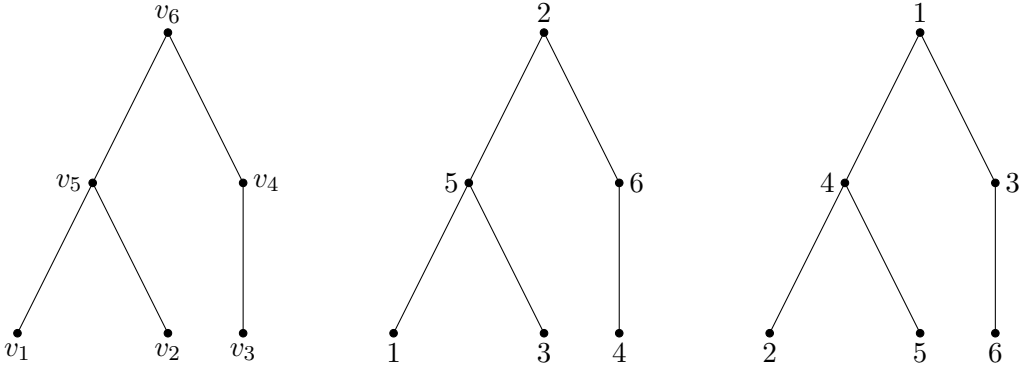

If $F$ is a linear path with $n$ vertices, $A_F(q)$ is equal to the $n$-th Eulerian polynomial $A_{n}(q)$, up to a factor of $q$.  $A_{n}(q)$  has only (negative and simple) real roots, a result due to Frobenius~\cite{frobenius1910uober}. As a consequence, the sequence of coefficients of $A_{n}(q)$ is log-concave and in turn unimodal. A general descent polynomial $A_F(q)$ may have non-real roots.  For example, for the tree $T$ in Figure \ref{fig: des poly example}, the descent polynomial $A_T(q) = 20q^5+90q^4+250q^3+250q^2+90q+20$ has only one real root, $q=-1$. Here we prove that the coefficients of $A_F(q)$ form a palindromic and unimodal sequence (Section \ref{sec: uni}). Then in Section \ref{sec: normal} we prove that the distribution of descents in a tree $T_{n}$ with $n$ vertices is asymptotically normal if the maximum down degree is $D_{n} = O(n^{0.5-\epsilon})$. We remark that if a tree $P_{n}$ is picked random from the set of $n^{n-2}$ trees with $n$ labeled nodes and $E(D)$ denotes the expected value of the maximum degree of $P_{n}$ then $E(D) \sim \frac{\log n}{ \log \log n}$ as $n \to \infty$~\cite{moon1968}.

The coefficient sequences of the descent polynomials of all the random trees that we have checked up to size 11 were all log-concave, but we don't have a proof of that.  Direct combinatorial proofs of the log-concavity of the Eulerian polynomials  have been gound  by Gasharov in \cite{GASHAROV1998} and by B\'ona and Ehrenborg in \cite{BONA2000}. Their proofs use bijections between permutations of size $n$ with $k$ descents and the set of labeled northeastern lattice paths with $n$ edges, exactly $k$ of which are vertical. Unfortunately, we don't see a way to extend that idea to our setting.

Generalizations of MacMahon's formula for $q$-counting inversions and major index in the setting of labelings of a fixed forest has been considered in~\cite{BJORNER1989}. For a graph $G = ([n], E)$ and a permutation $\sigma \in S_{n}$, define the graphical inversion number~\cite{FOATA1996} $\mathrm{inv}_{G}(\sigma)$ to be the number of edges $\{\sigma(i), \sigma(j)\} \in E$ such that $1 \leq i < j \leq n$ and
$\sigma(i) > \sigma(j)$. (So, $\inv_{G}$ generalizes the well-studied inversion statistic.) Suppose the vertex set of $F$ is $[n]$ and that $i < j$ whenever $i <_{F} j$. Counting the descents over all labelings of $F$ is equivalent to counting the graphical inversions over $S_{n}$ in the underlying graph $G$ of the Hasse diagram $F$.  When $G$ is the incomparability graph of a poset $P$ (an edge in $G$ corresponds to a pair of incomparable elements in $P$), the polynomial $A_{G}(q, p, t) := \sum_{\sigma \in S_{n}} t^{\inv_{G}(\sigma)}q^{\maj_{P} (\sigma)}p^{\des_{P} (\sigma)}$ has been studied in~\cite{shareshian2016chromatic}. On the other hand, we note that the count by descents and leaves in~\cite{gessel1996counting} as well as the tree Eulerian polynomial considered in~\cite{d2016note} is the sum over all labeled trees of size $n$ and therefore is different from the polynomials considered here.


\section{Symmetry and Unimodality}\label{sec: uni}

In this section we prove that the descent polynomials $A_{F}(q)$ is unimodal. The following result on the product of symmetric unimodal polynomials is used in the proof. A polynomial $f(x) = \sum_{i=0}^na_ix^i$ is \emph{symmetric}, or \emph{palindromic}, if $a_i=a_{n-i}$ for $i=0,1,\ldots ,n$.

\begin{proposition}[\cite{STANLEY1989}]\label{prop: unimodal product}
If $A(q)$ and $B(q)$ are symmetric unimodal polynomials with nonnegative coefficients, then so is $A(q)B(q)$.
\end{proposition}

First we show that $A_F(q)$ is symmetric.

\begin{lemma}\label{cor: symmetric}
For a forest $F$, the descent polynomial $A_F(q)$ is symmetric.
\end{lemma} 

\begin{proof}
Let $F$ be a forest with $n$ vertices and $m$ edges. For a labeling $w\in\mathcal{W}(F)$, define $w'\in\mathcal{W}(F)$ by $w'(x)=n+1-w(x)$. Then for a vertex $x$ with a child $y$, clearly $w(x) < w(y)$ if and only if $w'(x)>w'(y)$. Therefore every descent in $(F,w)$ corresponds to an ascent in $(F,w')$ and vice versa. Therefore, the number of labelings with $k$ descents is equal to the number of labelings with $m-k-1$ descents.
\end{proof}

Now we are ready for the main result of this section. In what follows, the \emph{down-degree} of a vertex $v$ is the number of children of $v$. For example, the down-degree of $v_{5}$ in the tree $T$ in Figure~\ref{fig: des poly example} is 2.

\begin{theorem}
For a forest $F$, the descent polynomial $A_F(q)$ is unimodal.
\end{theorem}

\begin{proof}
We will prove the claim by induction on the number of vertices in $F$. The base case is easy to check. Suppose $A_F(q)$ is unimodal for all forests $F$ of size less than  $n$. Consider a forest $F$ of size $n$ that consists of trees $T_1, \ldots, T_m$. We will consider the cases when $m>1$ and $m=1$ seperatly.

First suppose that $m>1$, or in other words, $F$ is not a tree. Then 
\[A_F(q)={n \choose k_1,\ldots,k_m}A_{T_1}(q)  \cdots A_{T_m}(q),\]
for $k_i$ is the size of $T_{i}$. By the inductive hypothesis, $A_{T_i}(q)$ is unimodal for all $i=1,\ldots,m$ and thus, by Proposition \ref{prop: unimodal product}, $A_{F}(q)$ is unimodal.

Now suppose that $m=1$, or in other words $F$, is a tree with $n$ vertices. 
For a vertex $v$,  let $F_v = F - v$ be the tree obtained by removing the vertex $v$ and incident edges from $F$. For $w \in \mathcal{W}(F)$, let $v$ be the vertex such that $w(v) = 1$. Consider the map defined by removing the vertex $v$ and its adjacent edges to get the forest $F_v$ and the labeling $w'\in \mathcal{W}(F_v)$ defined by $w'(x) = w(x)-1$ for all vertices $x$ in $F_v$ . This defines a bijection from $\mathcal{W}(F)$ to the set of pairs of a vertex $v$ and a labeling $w'\in \mathcal{W}(F_v)$. Notice that in $w$, $v$ creates a descent with all of its children since $w(v) = 1$ but it does not create a descent with its parent. So, we have

\begin{equation} \label{recurrence} A_F(q) = \sum_{w \in\mathcal{W}(F)}q^{\des(F, w)}= \sum_{v\in F} \sum_{w'\in\mathcal{W}(F_{v})} q^{\des(F_v, w')+d_v} = \sum_{v\in F} q^{d_{v}}A_{F_{v}}(q)\end{equation}
where $d_v$ is the down-degree of $v$.

Let $A_F(q) = a_0+a_1q+\cdots+a_{n-1}q^{n-1}$ and consider $a_k$ and $a_{k+1}$ for some $k+1\leq \lfloor \frac{n-1}{2}\rfloor$. Let $v_1,\ldots,v_{n+1}$ be the vertices of $F$, and let $A_{F_{v_i}}(q) = a_{i,0}+a_{i,1}q+a_{i,2}q^2+\cdots+a_{i,e_i}q^{e_i}$, where $e_i$ is the number of edges in $F_{v_i}$, i.e.,
\begin{equation} \label{edges} e_{i} = \begin{cases} n-1 -d_{v_{i}} \; \text{ if } v_{i}  \text{ is the root of } $F$\\
n - 2- d_{v_{i}} \; \text{ otherwise.}
\end{cases}\end{equation}
The coefficient $a_{i,j}$ is the number of labelings of the forest $F_{v_i}$ with $j$ descents and, in particular, $a_{i,j} =0$ if $j<0$. 
Using~\eqref{recurrence} we get,
\begin{equation}\label{ak} a_k=a_{1,k-d_{v_1}}+\cdots+a_{n,k-d_{v_n}} \end{equation}
 and, similarly,
 \begin{equation}\label{ak1} a_{k+1}=a_{1,k-d_{v_1}+1}+\cdots+a_{n,k-d_{v_n}+1}.\end{equation} Suppose that $k+1 \leq \lfloor \frac{n-1}{2}\rfloor$. We will show that $a_{k} \leq a_{k+1}$ by comparing the terms on the right hand-side of~\eqref{ak} and~\eqref{ak1} that correspond to the same forest $F_{v_{i}}$.
 Let us consider a fixed $i$. If $k-d_{v_i} < 0$ then  $a_{i,k-d_{v_i}} = 0 \leq a_{i,k-d_{v_i}+1}$. Let $k-d_{v_i} \geq 0$. If $v_{i}$ is not a leaf of $F$ then we have $d_{v_{i}}  \geq 1$ and 
 \[k-d_{v_i}+1 \leq \lfloor \frac{n-1}{2}\rfloor - d_{v_i}= \lfloor \frac{n - 1- 2d_{v_i} }{2}\rfloor  \leq  \lfloor \frac{e_{i} +1 -d_{v_{i}}}{2}\rfloor \leq \lfloor \frac{e_{i} }{2}\rfloor.\]  If $d_{v_{i}}  =0$, $v_{i}$ is a leaf and hence not a root of $F$. If additionally if $k +1 \leq \lfloor \frac{n-2}{2}\rfloor$, then we have  \[k-d_{v_i}+1 = k+1 \leq \lfloor \frac{n-2}{2}\rfloor = \lfloor \frac{e_{i} }{2}\rfloor.\] So, in both of these cases $a_{i,k-d_{v_i}}$ and $a_{i,k-d_{v_i}+1}$ are in the first half of the sequence $a_{i,0}, \ldots, a_{i,e_{i}}$ and, by the inductive hypothesis, $a_{i,k-d_{v_i}} \leq a_{i,k-d_{v_i}+1}$. Finally, if $d_{v_{i}} = 0$ and  $\lfloor \frac{n-2}{2}\rfloor < k+1 \leq \lfloor \frac{n-1}{2}\rfloor$ then $n$ is odd and $k+1 =  \frac{n-1}{2} = \frac{e_{i}+1}{2}$, $k = \frac{e_{i}-1}{2}$ and thus, by Lemma~\ref{cor: symmetric} we have $a_{i,k} =  a_{i,k+1}$, i.e., $a_{i,k-d_{v_i}} =  a_{i,k-d_{v_i}+1}$. 
\end{proof}

 \section{Central Limit Theorem}\label{sec: normal}

For a random variable $Z$, let \[\tilde{Z} = \frac{Z-\E(Z) }{\sqrt{\Var(Z)}}.\] We write $Z_n \to N(0,1)$ to mean that $Z_n$ converges in distribution to the standard normal distribution.

Consider the random variable $X_n$ which is counting the number of descents in a randomly generated labeling of a fixed tree $T_{n}$ of size $n$. In this section, we show  that under some assumptions on the maximum degrees of the trees $\{T_{n}\}_{n \geq 1}$, $\tilde{X}_n \to N(0,1)$ as $n \to \infty$.

\begin{theorem}\label{thm: main des}
Let $\{T_n\}_{n \geq 1}$ be a sequence of trees of size $n$ and  $X_n$ be the random variable that counts the number of descents in a random labeling of $T_n$. If $D_n \leq  Cn^{\frac{1}{2}-\epsilon}$ for some  constant $C$ and some $0<\epsilon<\frac{1}{2}$ where $D_n$ is the maximum down-degree in the tree $T_n$, then $\tilde{X_n} \to N(0,1)$.
\end{theorem}

The proof uses the Janson's dependency criterion~\cite{JANSON1988} that is stated in terms of a dependency graph as follows.  Let $\{Y_{k}\mid k=1,2,\ldots, N\}$ be a finite set of random variables. Then a graph $G$ is a dependency graph for $\{Y_{k}\mid k=1,2,\ldots, N\}$ if the following conditions are satisfied:
\begin{enumerate}
\item There exists a bijection between the random variables $\{Y_{k}\mid k=1,2,\ldots, N\}$ and the vertices of $G$, and
\item if $V_1$ and $V_2$ are disjoint sets of vertices of $G$ such that no edge of $G$ has one endpoint in $V_1$ and another one in $V_2$, then the corresponding sets of random variables are independent.
\end{enumerate}
Note that the dependency graph for a finite set of random variables is, in general, not unique because if the graph is not complete one can add another edge to obtain a new dependency graph. We can now state Janson's dependency criterion.

\begin{theorem}[\cite{JANSON1988}]\label{thm: Janson criteria}
Let $Y_{n,k}$ be an array of random variables such that for all $n$, and for all $k=1,\ldots,N_n$, the inequality $|Y_{n,k}| \leq A_n$ holds for some real number $A_n$, and that the maximum degree of a dependency graph of $\{Y_{n,k} \mid k=1,\ldots,N_n\}$ is $\Delta_n$. Set $Y_{n} = \sum_{k=1}^{N_n} Y_{n,k}$ and $\sigma_n^2 = \Var(Y_n)$. If there is a natural number $m$ so that 
\begin{equation}
N_n \Delta_n^{m-1} \left( \frac{A_n}{\sigma_n}\right)^m \to 0,
\end{equation}
then
\[\tilde{Y}_n \to N(0,1).\]
\end{theorem}

For each tree $T_{n}$ in the sequence, fix an ordering of its edges.  To apply Janson's criterion, let $Y_{n,k}$ be the indicator random variables $X_{n,k}$ of the event that the edge $k$ corresponds to a descent in a randomly selected labeling of $T_n$. Thus $N_n = n-1$, the number of edges in a tree of size $n$. By the definition of $Y_{n,k}$, we have $|Y_{n,k}|\leq 1$ so we will set $A_n = 1$.

Next we will look at a dependency graph $G$ for the random variables $X_{n,k}$ to get a bound on $\Delta_n$. The variables $X_{n,k_1}$ and $X_{n,k_2}$ are independent  if the edges $k_1$ and $k_2$ do not share a vertex. Therefore, we can take the dependency graph for $X_{n,k}$ to be the line graph of $T_{n}$: each vertex in  $G$ corresponds to an edge from the tree $T_n$ and two vertices of $G$ are adjacent if the corresponding edges of $T_n$ share an endpoint. Figure \ref{fig: dependency} shows this dependency graph for the tree $T$ in Figure \ref{fig: des poly example}. Let $D_n$ denote the largest down-degree of a vertex in the tree $T$, then $\Delta_n \leq 2D_n$.

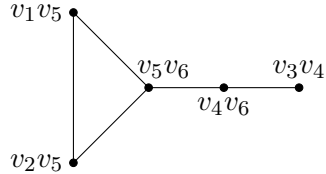
\begin{figure}[h]
\centering
\begin{tikzpicture}
\draw (0,0)--(0,2)--(1,1)--(2,1)--(3,1);
\draw (0,0)--(1,1);
\draw [fill] (0,0) circle [radius=0.05];
\draw [fill] (0,2) circle [radius=0.05];
\draw [fill] (1,1) circle [radius=0.05];
\draw [fill] (2,1) circle [radius=0.05];
\draw [fill] (3,1) circle [radius=0.05];
\node [left] at (0,0) {$v_2v_5$};
\node [left] at (0,2) {$v_1v_5$};
\node [above] at (1.2,1) {$v_5v_6$};
\node [below] at (2,1) {$v_4v_6$};
\node [above] at (3,1) {$v_3v_4$};
\end{tikzpicture}
\caption{Dependency graph $G$ of the tree $T$ in Figure \ref{fig: des poly example}} \label{fig: dependency}   
\end{figure}

\begin{lemma}\label{lem: var}
For a tree $T$ of size $n$ with root $r$,
\[\E(X_n) = \frac{n-1}{2} \hspace{1cm} \text{and} \hspace{1cm}  \Var(X_n) = \frac{2d_r+\sum_{v \in T} d_v^2}{12},\] where $d_v$ is the down-degree of vertex $v$.
\end{lemma}

\begin{proof}
Fix an ordering of the edges of the tree $T$ and let  $X_{n,k}$, $k=1, \cdots, n-1$ be the indicator random variable for whether the vertices on edge $k$ create a descent. Then $X_{n}  = \sum_{k=1}^{n-1} X_{n,k}$. $\E(X_{n,k}) = \frac{1}{2}$ because, as we saw in the proof of Lemma~\ref{cor: symmetric}, two vertices on a same edge create a descent in half of the labelings of a given tree $T$, so the formula for $\E(X_n)$ easily follows.
\begin{align}
\Var(X_n) & = \E(X_n^2)-(\E(X_n)^2)\\
& = \E\left(\left(\sum_{k=1}^{n-1} X_{n,k}\right)^2\right) - \left(\E\left(\sum_{k=1}^{n-1} X_{n,k}\right)\right)^2\\
& = \E\left(\left(\sum_{k=1}^{n-1} X_{n,k}\right)^2\right) - \left(\sum_{k=1}^{n-1}\E(X_{n,k})\right)^2\\
& = \sum_{k_1,k_2} \E(X_{n,k_1}X_{n,k_2}) - \sum_{k_1,k_2}\E(X_{n,k_1})\E(X_{n,k_2}),\label{eq: var}
\end{align}
where the last two sums run over all ordered pairs $(k_1,k_2)\in \{1,\ldots,n-1\}\times \{1,\ldots,n-1\}$.

Now, since $\E(X_{n,k}) = \frac{1}{2}$, the $\E(X_{n,k_1})\E(X_{n,k_2})$ terms appearing in \eqref{eq: var} are all equal to $\frac{1}{4}$. We will now calculate the values for the $\E(X_{n,k_1}X_{n,k_2})$ terms in \eqref{eq: var}. If the edges $k_1$ and $k_2$ do not share a vertex, then they are independent and we get $\E(X_{n,k_1}X_{n,k_2}) = \E(X_{n,k_1})\E(X_{n,k_2}) = \frac{1}{4}$, and if $k_1 = k_2$ then $\E(X_{n,k_1}X_{n,k_2}) = \E(X_{n,k_1}^2) = \E(X_{n,k_1}) = \frac{1}{2}$. Let $k_1$ be the edge $v_iv_j$ and $k_2$ be the edge $v_sv_t$, with $i <j$ and $s<t$. There are three cases left to consider: if $j=t$ we have the case shown in Figure \ref{subfig: a}, if $j=s$ we have the case in Figure \ref{subfig: b}, and if $i=t$ we have the case shown in Figure \ref{subfig: c}. If $j=t$, then $\E(X_{n,k_1}X_{n,k_2}) = \frac{1}{3}$ since $X_{n,k_1}=X_{n,k_2}=1$ if and only if $w(v_i)<w(v_j)$ and $w(v_i)<w(v_s)$. There are 6 ways to order fixed values of the three labels and two of them satisfy that requirement, so $\E(X_{n,k_1}X_{n,k_2}) = \frac{2}{6}=\frac{1}{3}$. Similarly, if $j=s$ or $i=t$ we see $\E(X_{n,k_1}X_{n,k_2}) = \frac{1}{6}$.

Next, we will count how many of the terms $\E(X_{n,k_1}X_{n,k_2})$ in~\eqref{eq: var} are $\frac{1}{2}$, $\frac{1}{3}$, or $\frac{1}{6}$.
\begin{itemize}
\item We know that $\E(X_{n,k_1}X_{n,k_2}) = \frac{1}{2}$ only when $k_1=k_2$ and therefore this occurs $n-1$ times, once for each edge.
\item If $\E(X_{n,k_1}X_{n,k_2})=\frac{1}{3}$, we have the case from Figure \ref{subfig: a}.  For each vertex $v$ in the tree $T$ this will occur $d_v(d_v-1)$ times since we have $d_v$ choices for the first child and then $d_v-1$ choices for the second. This occurs a total of $\sum_{v\in  T}d_v(d_v-1)$ times in the tree $T$.
\item Lastly, if $\E(X_{n,k_1}X_{n,k_2})=\frac{1}{6}$, we have the cases shown in Figures \ref{subfig: b} and \ref{subfig: c}. If $r$ is the root of $T$, then for each vertex $v \neq r$ in $T$ this occurs $2d_v$ times. There are $d_{v}$ choices for the lower edge and one choice for the upper edge, and the edges could appear in either order in the product. Therefore this case appears $\sum_{v \neq r} 2d_v$ times throughout the tree $T$.
\end{itemize}
Plugging this information into \eqref{eq: var}, we get 
\begin{align}
\Var(X_n) & = \left(\frac{1}{2}-\frac{1}{4}\right)(n-1) + \left(\frac{1}{3}-\frac{1}{4}\right)\sum_{v\in T}d_v(d_v-1)+\left(\frac{1}{6}-\frac{1}{4}\right)\sum_{\substack{ v\neq r}}2d_v\\
&= \frac{1}{4}(n-1)+\frac{1}{12}\sum_{v\in T}d_v(d_v-1) -\frac{1}{12}\sum_{\substack{ v\neq r}}2d_v\\
& = \frac{3(n-1)+2d_r+\sum_{v \in T} d_v^2 -3\sum_{v\in T} d_v }{12}\\
& = \frac{2d_r+\sum_{v\in T}d_v^2 }{12}.
\end{align}
\end{proof}

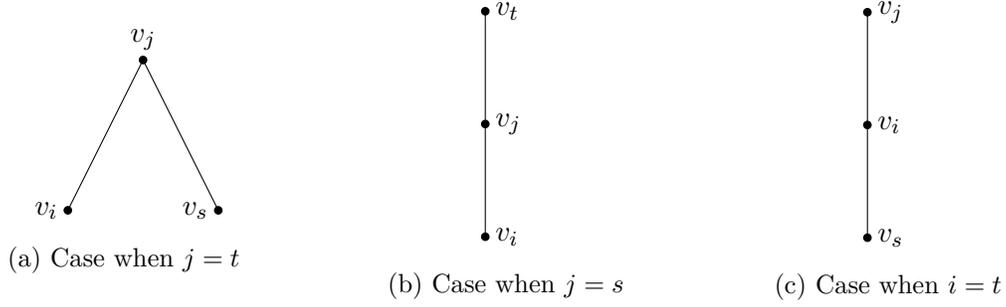
\begin{figure}
\centering
\begin{subfigure}{0.3\textwidth}
\centering
\begin{tikzpicture}
\draw (0,0)--(1,2)--(2,0);
\draw [fill] (0,0) circle [radius=0.05];
\draw [fill] (1,2) circle [radius=0.05];
\draw [fill] (2,0) circle [radius=0.05];
\node [left] at (0,0) {$v_i$};
\node [above] at (1,2) {$v_j$};
\node [left] at (2,0) {$v_s$};
\end{tikzpicture}
\caption{Case when $j=t$}\label{subfig: a}
\end{subfigure}
\begin{subfigure}{0.3\textwidth}
\centering
\begin{tikzpicture}
\draw [fill] (0,-.5) circle [radius=0.05];
\draw [fill] (0,1) circle [radius=0.05];
\draw [fill] (0,2.5) circle [radius=0.05];
\node [right] at (0,2.5) {$v_t$};
\node [right] at (0,1) {$v_j$};
\node [right] at (0,-.5) {$v_i$};
\draw (0,-.5)--(0,2.5);
\end{tikzpicture}
\caption{Case when $j=s$}\label{subfig: b}
\end{subfigure}
\begin{subfigure}{0.3\textwidth}
\centering
\begin{tikzpicture}
\draw [fill] (0,-.5) circle [radius=0.05];
\draw [fill] (0,1) circle [radius=0.05];
\draw [fill] (0,2.5) circle [radius=0.05];
\node [right] at (0,2.5) {$v_j$};
\node [right] at (0,1) {$v_i$};
\node [right] at (0,-.5) {$v_s$};
\draw (0,-.5)--(0,2.5);
\end{tikzpicture}
\caption{Case when $i=t$}\label{subfig: c} 
\end{subfigure}
\caption{The cases where the edges $k_1=v_iv_j$ and $k_2=v_sv_t$ are adjacent} \label{fig: expectation example}
\end{figure}

Using the variance calculated in Lemma~\ref{lem: var} and the values we found for $N_n$, $\Delta_n$, and $A_n$ we can now apply Janson's criterion to prove Theorem~ \ref{thm: main des}.

\begin{proof}[Proof of Theorem~\ref{thm: main des}]
In Lemma \ref{lem: var}, we showed that $\Var(X_n) = \frac{2d_r+\sum_{v\in T_{n}}d_v^2 }{12}$. At least one of the vertices in the tree must have down-degree $D_n$. If $v^*$ is one such vertex, then we get

\begin{align}\label{eq: var estimate}
\begin{split}
\Var(X_n) & =  \frac{2d_r+\sum_{v\in T_{n}}d_v^2 }{12} = \frac{2d_r+ D_n^2 + \sum_{v\neq v^*}d_v^2}{12} \geq \frac{2d_r+ D_n^2 + \sum_{v\neq v^*}d_v}{12} \\
&= \frac{2d_r+D_n^2+(n-1-D_n)}{12} \geq \frac{D_n^2+n-1-D_n+2}{12} = \frac{n+D_n^2-D_n+1}{12}.
\end{split}
\end{align}
Note that this bound is tight when the maximum degree does not appear at the root, and the rest of the vertices have down-degree one. 

To apply Janson's criterion with $N_n = n-1$, $\Delta_n \leq 2D_n$, $A_n = 1$, and the estimate \eqref{eq: var estimate}, we need to show there is a natural number $m$ such that 
\begin{equation*}
(n-1)(2D_n)^{m-1} \left(\frac{12}{n+D_n^2-D_n+1}\right)^{\frac{m}{2}} \to 0.
\end{equation*}
It suffices to show that there is a natural number $m$ such that 
\begin{equation}\label{eq: lim}
\frac{nD_n^{m-1}}{(n+D_n^2-D_n+1)^{\frac{m}{2}}} = \frac{\frac{nD_n^{m-1}}{n^\frac{m}{2}}} { (1+\frac{D_n^2}{n} -\frac{D_n}{n} + \frac{1}{n})^{\frac{m}{2}}}   \to 0.
\end{equation}
Under the assumptions about the growth the maximal degrees, we have 
\begin{equation*}
\frac{D_n}{n}\leq \frac{D_n^2}{n} \leq \frac {C^2(n^{\frac{1}{2}-\epsilon})^2}{n} \to 0
\end{equation*}
for $0<\epsilon<\frac{1}{2}$.
This means that 
\begin{equation}\label{eq: num}
\frac{nD_n^{m-1}}{n^{\frac{m}{2}}} = \frac{D_n^{m-1}}{n^{\frac{m}{2}-1}} \leq \frac{C^{m-1}n^{(\frac{1}{2}-\epsilon)(m-1)}}{n^{\frac{m}{2}-1}} \to 0
\end{equation}
if $\frac{m}{2}-1>(\frac{1}{2}-\epsilon)(m-1)=\frac{1}{2}m-\frac{1}{2}-\epsilon(m-1)$. In other words, \eqref{eq: num} holds if $\epsilon > \frac{1}{2(m-1)}$. 

Note that $\frac{1}{2(m-1)}$ approaches 0 as $m\to\infty$ and thus for any $0<\epsilon<\frac{1}{2}$, there is some $m\geq 3$ such that \eqref{eq: lim} holds. Therefore, by Janson's criterion, we get that $\tilde{X_n} \to N(0,1)$. 
\end{proof}

\bibliographystyle{plain}

\begin{thebibliography}{10}

\bibitem{BJORNER1989}
A.~Bj{\"o}rner and M.L. Wachs.
\newblock q-{H}ook length formulas for forests.
\newblock {\em J. Combin. Theory Ser. A}, 52(2):165 -- 187, 1989.

\bibitem{BONA2000}
M.~B{\'o}na and R.~Ehrenborg.
\newblock A combinatorial proof of the log-concavity of the numbers of
  permutations with k runs.
\newblock {\em J. Combin. Theory Ser. A}, 90(2):293 -- 303, 2000.

\bibitem{d2016note}
Gonz{\'a}lez D'Le{\'o}n and S~Rafael.
\newblock A {N}ote on the $\gamma$-coefficients of the tree
  {E}ulerian polynomial.
\newblock {\em The Electronic Journal of Combinatorics}, 23(1):P1.20, 2016.

\bibitem{FOATA1996}
D.~Foata and D.~Zeilberger.
\newblock Graphical major indices.
\newblock {\em J. Comput. Math.}, 68(1):79 -- 101, 1996.

\bibitem{frobenius1910uober}
G.~Frobenius.
\newblock {\:U}ber die {B}ernoullischen und die {E}ulerschen {P}olynome.
\newblock {\em Sitzungsberichte der Preussische Akademie der Wissenschaften},
  809--847, 1910.

\bibitem{GASHAROV1998}
V.~Gasharov.
\newblock On the {N}eggers--{S}tanley conjecture and the {E}ulerian
  polynomials.
\newblock {\em J. Combin. Theory Ser. A}, 82(2):134 -- 146, 1998.

\bibitem{gessel1996counting}
Ira Gessel.
\newblock Counting forests by descents and leaves.
\newblock {\em The Electronic Journal of Combinatorics}, 3(2), Research paper \#5, 1996.

\bibitem{JANSON1988}
S.~Janson.
\newblock Normal convergence by higher semiinvariants with applications to sums
  of dependent random variables and random graphs.
\newblock {\em Ann. Probab.}, 16(1):305--312, 1988.

\bibitem{moon1968}
J.~W. Moon.
\newblock On the maximum degree in a random tree.
\newblock {\em Michigan Math. J.}, 15(4):429--432, 1968.

\bibitem{shareshian2016chromatic}
John Shareshian and Michelle~L Wachs.
\newblock Chromatic quasisymmetric functions.
\newblock {\em Advances in Mathematics}, 295:497--551, 2016.

\bibitem{STANLEY1989}
R.P. Stanley.
\newblock Log-{C}oncave and {U}nimodal {S}equences in {A}lgebra,
  {C}ombinatorics, and {G}eometry.
\newblock {\em Ann. N.Y. Acad. Sci.}, 576(1):500--535, 1989.

\end{thebibliography}

\end{document}